\documentclass[11pt]{amsart}
\usepackage{inputenc}
\usepackage{hyperref}
\usepackage[a4paper,hmargin=3.5cm,vmargin=4cm]{geometry}
\usepackage{amsfonts,amssymb,amscd,amstext,verbatim}

\usepackage{fancyhdr}
\usepackage{times}

\usepackage{enumerate}
\usepackage{titlesec}
\usepackage{mathrsfs}

\titleformat{\section}
{\filcenter\bfseries\large} {\thesection{.}}{0.2cm}{}
\titleformat{\subsection}[runin]
{\bfseries} {\thesubsection{.}}{0.15cm}{}[.]
\titleformat{\subsubsection}[runin]
{\em}{\thesubsubsection{.}}{0.15cm}{}[.]

\usepackage[up,bf]{caption}

\newtheorem{theorem}{Theorem}[section]
\newtheorem{proposition}[theorem]{Proposition}

\newtheorem{corollary}[theorem]{Corollary}

\theoremstyle{definition}
\newtheorem{definition}[theorem]{Definition}

\newtheorem{problem}[theorem]{Problem}

\numberwithin{equation}{section}
\numberwithin{figure}{section}

\newcommand\Ocal{\mathcal{O}}

\newcommand\Cscr{\mathscr{C}}

\newcommand\Hscr{\mathscr{H}}

\newcommand\B{\mathbb{B}}
\newcommand\C{\mathbb{C}}
\newcommand\CP{\mathbb{CP}}
\newcommand\D{\mathbb D}

\newcommand\N{\mathbb{N}}
\renewcommand\P{\mathbb{P}}
\newcommand\R{\mathbb{R}}

\newcommand\Z{\mathbb{Z}}

\newcommand\igot{\mathfrak{i}}

\renewcommand\igot{\mathfrak{i}}

\renewcommand\imath{\igot}

\newcommand\hra{\hookrightarrow}

\newcommand\wt{\widetilde}
\newcommand\wh{\widehat}

\newcommand\dibar{\overline\partial}

\newcommand\Aut{\mathrm{Aut}}

\begin{document}

\fancyhead[LO]{Holomorphic embeddings and immersions of Stein manifolds: a survey}
\fancyhead[RE]{F.\ Forstneri\v c} 
\fancyhead[RO,LE]{\thepage}

\thispagestyle{empty}

\vspace*{1cm}
\begin{center}
{\bf\LARGE Holomorphic embeddings and immersions \\ of Stein manifolds: a survey}

\vspace*{0.4cm}

{\large\bf  Franc Forstneri\v c} 

\vspace*{4mm}
{\em Dedicated to Kang-Tae Kim for his sixtieth birthday}
\end{center}


\vspace*{1cm}

\begin{quote}
{\small
\noindent {\bf Abstract}\hspace*{0.1cm}
In this paper we survey results on the existence of holomorphic embeddings and
immersions of Stein manifolds into complex manifolds.  Most of them pertain to proper maps
into Stein manifolds. We include a new result saying that 
every continuous map $X\to Y$ between Stein manifolds is homotopic to a proper
holomorphic embedding provided that $\dim Y>2\dim X$ and 
we allow a homotopic deformation of the Stein structure on $X$. 

\vspace*{0.2cm}

\noindent{\bf Keywords}\hspace*{0.1cm} Stein manifold, embedding, density property, Oka manifold

\vspace*{0.1cm}


\noindent{\bf MSC (2010):}\hspace*{0.1cm}}  32E10, 32F10, 32H02, 32M17, 32Q28, 58E20, 14C30
%
%
\end{quote}


\section{Introduction} 
\label{sec:intro}

In this paper we review what we know about the existence of holomorphic embeddings
and immersions of Stein manifolds into other complex manifolds. The emphasis is on recent results, 
but we also include some classical ones for the sake of completeness and historical perspective. 
Recall that Stein manifolds are precisely the closed complex submanifolds of Euclidean spaces $\C^N$ 
(see Remmert \cite{Remmert1956},  Bishop \cite{Bishop1961AJM}, and Narasimhan \cite{Narasimhan1960AJM};
cf.\ Theorem \ref{th:classical}). Stein manifolds of dimension $1$ are open Riemann surfaces 
(see Behnke and Stein \cite{BehnkeStein1949}). A domain in $\C^n$ is Stein if and only if it is a domain of 
holomorphy (see Cartan and Thullen \cite{CartanThullen1932}). For more information, see the monographs \cite{Forstneric2017E,GrauertRemmert1979,GunningRossi2009,HormanderSCV}.  

In \S \ref{sec:Euclidean} we survey results on the existence of proper holomorphic immersions and embeddings 
of Stein manifolds into Euclidean spaces. 
Of special interest are the minimal embedding and immersion dimensions. 
Theorem \ref{th:EGS}, due to Eliashberg and Gromov \cite{EliashbergGromov1992AM} (1992)
and Sch\"urmann \cite{Schurmann1997} (1997), settles this question for Stein manifolds of dimension $>1$. 
It remains an open problem whether every open Riemann surface embeds holomorphically 
into $\C^2$; we describe its current status in \S \ref{ss:RS}. We also discuss the use of holomorphic 
automorphisms of Euclidean spaces in the construction of wild holomorphic embeddings  
(see \S \ref{ss:wild} and \S\ref{ss:complete}).

It has recently been discovered by Andrist et al. \cite{AndristFRW2016,AndristWold2014,Forstneric-immersions}
that there is a big class of Stein manifolds  $Y$ which contain every Stein manifold $X$ with $2\dim X< \dim Y$ 
as a closed complex submanifold (see Theorem \ref{th:density} ).
In fact, this holds for every Stein manifold $Y$ enjoying  Varolin's {\em density property}  \cite{Varolin2000,Varolin2001}:
the Lie algebra of all holomorphic vector fields on $Y$ is spanned by the $\C$-complete vector fields, 
i.e., those whose flow is an action of the additive group $(\C,+)$ by holomorphic automorphisms
of $Y$ (see Definition \ref{def:density}).  Since the domain $(\C^*)^n$ enjoys the volume density property,
we infer that every Stein manifold $X$ of dimension $n$ admits a proper holomorphic 
immersion to $(\C^*)^{2n}$ and a proper pluriharmonic map into $\R^{2n}$
(see Corollary \ref{cor:harmonic}). This provides a counterexample to the 
Schoen-Yau conjecture \cite{SchoenYau1997} for any  Stein  source manifold (see \S\ref{ss:S-Y}).

The class of Stein manifolds (in particular, of affine algebraic manifolds) with the density property is 
quite big and contains most complex Lie groups and homogeneous spaces,
as well as many nonhomogeneus manifolds.
This class has been the focus of intensive research during the last decade; we refer the reader to
the recent surveys \cite{KalimanKutzschebauch2015} and \cite[\S 4.10]{Forstneric2017E}. 
An open problem posed by Varolin \cite{Varolin2000,Varolin2001}
is whether every contractible Stein manifold with the density property is
biholomorphic to a Euclidean space. 

In \S \ref{sec:PSC} we recall a result of Drinovec Drnov\v sek and  
the author \cite{DrinovecForstneric2007DMJ,DrinovecForstneric2010AJM} 
to the effect that every smoothly bounded, strongly pseudoconvex Stein domain $X$ embeds properly
holomorphically into an arbitrary Stein manifold $Y$ with $\dim Y>2\dim X$. 
More precisely, every continuous map $\overline X\to Y$ which is holomorphic on $X$
is homotopic to a proper holomorphic embedding  $X\hra Y$ (see Theorem \ref{th:BDF2010}). 
The analogous result holds for immersions if $\dim Y\ge 2\dim X$, and
also for every $q$-complete manifold $Y$ with 
$q\in \{1,\ldots,\dim Y-2\dim X+1\}$, where the Stein case 
corresponds to $q=1$. This summarizes a long line of previous results.
In \S \ref{ss:Hodge} we mention a recent application of these techniques 
to the {\em Hodge conjecture} for the highest dimensional a priori nontrivial cohomology 
group of a $q$-complete manifold \cite{FSS2016}.
In \S\ref{ss:complete} we survey recent results on the existence of {\em complete}
proper holomorphic embeddings and immersions of strongly pseudoconvex domains into balls.
Recall that a submanifold of $\C^N$ is said to be {\em complete} if every divergent curve in it has
infinite Euclidean length.

In \S \ref{sec:soft} we show how the combination of the techniques
from \cite{DrinovecForstneric2007DMJ,DrinovecForstneric2010AJM} 
with those of Slapar and the author \cite{ForstnericSlapar2007MRL,ForstnericSlapar2007MZ} 
can be used to prove that, if $X$ and $Y$ are Stein manifolds and $\dim Y>2\dim X$, 
then every continuous map $X\to Y$ is homotopic to a proper
holomorphic embedding up to a homotopic deformation of the Stein structure on $X$  (see Theorem \ref{th:soft}).
The analogous result holds for immersions if $\dim Y\ge 2\dim X$, and 
for $q$-complete manifolds $Y$ with $q\le \dim Y-2\dim X+1$.
A result in a similar vein, concerning proper holomorphic 
embeddings of open Riemann surfaces into $\C^2$ up to a deformation
of their conformal structures, is due to Alarc\'on and L{\'o}pez \cite{AlarconLopez2013}
(a special case was proved in \cite{CerneForstneric2002}); see also 
Ritter \cite{Ritter2014} for embeddings into $(\C^*)^2$.

I have not included any topics from Cauchy-Riemann geometry 
since it would be impossible to properly discuss this major subject in the 
present survey of limited size and with a rather different focus.
The reader may wish to consult the recent survey by Pinchuk et al.\ \cite{Pinchuk2017},
the monograph by Baouendi et al.\ \cite{Baouendi1999} from 1999, and 
my survey \cite{Forstneric1993MN} from 1993.  For a new direction in this field,
see the papers by Bracci and Gaussier \cite{BracciGaussier2016X,BracciGaussier2017X}.

We shall be using the following notation and terminology. Let $\N=\{1,2,3,\ldots\}$.
We denote by $\D=\{z\in \C:|z|<1\}$ the unit disc in $\C$, by $\D^n\subset\C^n$ 
the Cartesian product of $n$ copies of $\D$ (the unit polydisc in $\C^n$), and by 
$\B^n=\{z=(z_1,\ldots,z_n)\in\C^n : |z|^2 = |z_1|^2+\cdots +|z_n|^2<1\}$ the unit ball in $\C^n$.
By $\Ocal(X)$ we denote the algebra of all holomorphic functions 
on a complex manifold $X$, and by $\Ocal(X,Y)$ the space of all holomorphic maps 
$X\to Y$ between a pair of complex manifolds; thus $\Ocal(X)=\Ocal(X,\C)$. These spaces carry the compact-open topology. 
This topology can be defined by a complete metric which renders them Baire spaces; 
in particular, $\Ocal(X)$ is a Fr\'echet algebra. (See \cite[p.\ 5]{Forstneric2017E} for more details.)
A compact set $K$ in a complex manifold $X$ is said to be {\em $\Ocal(X)$-convex} if 
$K=\wh K:= \{p\in X : |f(p)|\le \sup_K |f|\  \text{for every} \ f\in \Ocal(X)\}$.


\section{Embeddings and immersions of Stein manifolds into Euclidean spaces} 
\label{sec:Euclidean}

In this section we survey  results on proper holomorphic immersions and embeddings of Stein manifolds into 
Euclidean spaces. 

%
%
\subsection{Classical results}\label{ss:classical}
We begin by recalling the results of Remmert \cite{Remmert1956},  
Bishop \cite{Bishop1961AJM}, and Narasimhan \cite{Narasimhan1960AJM}
from the period 1956--1961.

\begin{theorem}\label{th:classical} 
Assume that $X$ is a Stein manifold of dimension $n$.
\begin{itemize}
\item[\rm (a)] If $N>2n$ then the set of proper embeddings $X\hra \C^N$ is dense in $\Ocal(X,\C^N)$.
\item[\rm (b)] If $N\ge 2n$ then the set of proper immersions $X\hra \C^N$ is dense in $\Ocal(X,\C^N)$.
\item[\rm (c)] If $N>n$ then the set of proper maps $X\to \C^N$ is dense in $\Ocal(X,\C^N)$.
\item[\rm (d)] If $N\ge n$ then the set of almost proper maps $X\to \C^N$ is residual in $\Ocal(X,\C^N)$.
\end{itemize}
\end{theorem}

A proof  of these results can also be found in the monograph by Gunning and Rossi \cite{GunningRossi2009}.

Recall that a set in a Baire space (such as $\Ocal(X,\C^N)$) is said to be {\em residual}, or a set
{\em of second category}, if it is the intersection of at most countably many open everywhere dense sets.
Every residual set is dense. A property of elements in a Baire space is said to be {\em generic} if it 
holds for all elements in a residual set.

The density statement for embeddings and immersions
is an easy consequence of the following result which follows from the jet transversality theorem 
for holomorphic maps. (See Forster \cite{Forster1970} for maps to Euclidean spaces and Kaliman and Zaidenberg 
\cite{KalimanZaidenberg1996TAMS} for the general case. A more complete
discussion of this topic can be found in \cite[\S 8.8]{Forstneric2017E}.) 
Note also that maps which are immersions or embeddings on a 
given compact set constitute an open set in the corresponding mapping space.

\begin{proposition}\label{prop:generic}
Assume that $X$ is a Stein manifold, $K$ is a compact set in $X$, and $U\Subset X$ is an open 
relatively compact set containing $K$. If $Y$ is a complex manifold such that
$\dim Y>2\dim X$, then every holomorphic map $f\colon X\to Y$ can be approximated uniformly
on $K$ by holomorphic embeddings $U\hra Y$. If $2\dim X \le \dim Y$ then $f$ can be approximated
by holomorphic immersions $U\to Y$.
\end{proposition}

Proposition \ref{prop:generic} fails in general without shrinking the domain of the map,
for otherwise it would yield nonconstant holomorphic maps of $\C$ to any complex manifold
of dimension $>1$ which is clearly false. On the other hand, it holds without shrinking the domain of the map
if the target manifold $Y$ satisfies a suitable holomorphic flexibility property, in particular, if it is an {\em Oka manifold}.
See \cite[Chap.\ 5]{Forstneric2017E} for the definition of this class of complex manifolds
and \cite[Corollary 8.8.7]{Forstneric2017E} for the mentioned result.

In the proof of Theorem \ref{th:classical}, parts (a)--(c),
we exhaust $X$ by a sequence $K_1\subset K_2\subset \cdots$ of compact $\Ocal(X)$-convex 
sets and approximate  the holomorphic map $f_j\colon X\to\C^N$ in the inductive step, uniformly on $K_j$, 
by a holomorphic map $f_{j+1}\colon X\to\C^N$ whose norm $|f_{j+1}|$ is not too small on $K_{j+1}\setminus K_j$
and such that $|f_{j+1}(x)|>1+\sup_{K_j} |f_j|$ holds for all $x\in bK_{j+1}$. If  the approximation is close enough at every step
then the sequence $f_j$ converges to a proper holomorphic map $f=\lim_{j\to\infty} f_j\colon X\to\C^N$.
If $N>2n$ then every map $f_j$ in the sequence can be made an embedding on $K_{j}$
(immersion in $N\ge 2n$) by Proposition \ref{prop:generic}, and hence the limit map $f$ is also such.

A more efficient way of constructing proper maps, immersions and embeddings of Stein manifolds into
Euclidean space was introduced by Bishop \cite{Bishop1961AJM}. He showed that any holomorphic map $X\to\C^n$ 
from an $n$-dimensional Stein manifold $X$ can be approximated uniformly on compacts by 
{\em almost proper} holomorphic maps $h\colon X\to \C^n$; see Theorem \ref{th:classical}(d).
More precisely, there is an increasing sequence $P_1\subset P_2\subset \cdots\subset X$ 
of relatively compact open sets exhausting $X$ such that every $P_j$ 
is a union of finitely many special analytic polyhedra 
and $h$ maps $P_j$ properly onto a polydisc $a_j \D^n\subset \C^n$, 
where $0<a_1<a_2<\ldots$ and $\lim_{j\to\infty} a_j=+\infty$.
We then obtain a proper map $(h,g)\colon X\to \C^{n+1}$ by choosing  $g\in \Ocal(X)$
such that for every $j\in\N$ we have $g>j$ on the compact set $L_j=\{x\in \overline P_{j+1}\setminus P_j : |h(x)|\le a_{j-1}\}$;
since $\overline P_{j-1}\cup L_j$ is $\Ocal(X)$-convex, this is possible by inductively using the Oka-Weil theorem. 
One can then find proper immersions and embeddings by adding a suitable number of additional components 
to $(h,g)$ (any such map is clearly proper) and using Proposition \ref{prop:generic} and the 
Oka-Weil theorem inductively. 

The first of the above mentioned procedures easily adapts to give a proof of the following interpolation theorem due to 
Acquistapace et al.\ \cite[Theorem 1]{Acquistapace1975}. Their result also pertains
to Stein spaces of bounded embedding dimension.

%
%
%
%
\begin{theorem} \label{th:ABT}
{\rm \cite[Theorem 1]{Acquistapace1975}}
Assume that $X$ is an $n$-dimensional  Stein manifold, $X'$ is a closed complex 
subvariety of $X$, and $\phi\colon X'\hra \C^N$ is a proper holomorphic 
embedding  for some $N > 2n$. Then the set of all proper holomorphic embeddings 
$X\hra \C^N$ that extend $\phi$ is  dense in the space of all holomorphic maps $X \to\C^N$ extending 
$\phi$. The analogous result holds for proper holomorphic immersions $X\to \C^N$ when $N\ge 2n$.
\end{theorem}

This interpolation theorem fails when $N<2n$. Indeed, for every $n>1$ there exists a proper holomorphic   
embedding $\phi \colon \C^{n-1} \hra \C^{2n-1}$ such that $\C^{2n-1}\setminus \phi(\C^{n-1})$
is Eisenman $n$-hyperbolic, so $\phi$ does not extend to an injective holomorphic map $f\colon \C^n\to \C^{2n-1}$
(see \cite[Proposition 9.5.6]{Forstneric2017E}; this topic is discussed in \S\ref{ss:wild}).  
The answer to the interpolation problem for embeddings seems unknown in the borderline case $N=2n$.

%
%
\subsection{Embeddings and immersions into spaces of minimal dimension}
\label{ss:minimal}
After Theorem \ref{th:classical} was proved in the early 1960's, one of the main questions 
driving this theory during the next decades was to find the smallest number $N=N(n)$ 
such that every Stein manifold $X$ of dimension $n$ embeds or immerses
properly holomorphically into $\C^N$. The belief that a Stein manifold of complex dimension $n$ admits proper holomorphic 
embeddings to Euclidean spaces of dimension smaller than $2n+1$ was based on the observation
that such a manifold is homotopically equivalent to a CW complex of dimension 
at most $n$; this follows from Morse theory 
(see Milnor \cite{Milnor1963}) and the existence of strongly plurisubharmonic Morse exhaustion functions 
on $X$ (see Hamm \cite{Hamm1983} and \cite[\S 3.12]{Forstneric2017E}). This problem, which was investigated 
by Forster \cite{Forster1970}, Eliashberg and Gromov \cite{EliashbergGromov1971}
and others, gave rise to major new methods in Stein geometry. 
Except in the case $n=1$ when $X$ is an open Riemann surface, 
the following optimal answer was given by Eliashberg and Gromov  \cite{EliashbergGromov1992AM} in 1992,
with an improvement by one for odd values on $n$ due to Sch{\"u}rmann \cite{Schurmann1997}.

%
%
%
%
\begin{theorem}
\label{th:EGS} 
{\em \cite{EliashbergGromov1992AM,Schurmann1997}}
Every Stein manifold $X$ of dimension $n$ immerses properly holomorphically into
$\C^M$ with $M = \left[\frac{3n+1}{2}\right]$, and if $n>1$ then $X$ embeds properly holomorphically
into $\C^N$ with $N = \left[\frac{3n}{2}\right]+ 1$.
\end{theorem}

Sch\"urmann \cite{Schurmann1997} also found optimal embedding dimensions for Stein spaces with singularities 
and with bounded embedding dimension.

The key ingredient in the proof of Theorem \ref{th:EGS} 
is a certain major extension of the Oka-Grauert theory, due to Gromov 
whose 1989 paper \cite{Gromov1989} marks the beginning of {\em modern Oka theory}. 
(See \cite{ForstnericLarusson2011} for an introduction to Oka theory and 
\cite{Forstneric2017E} for a complete account.)

Forster  showed in \cite[Proposition 3]{Forster1970} that the embedding dimension 
$N=\left[\frac{3n}{2}\right]+ 1$ is the minimal possible for every $n>1$, and the immersion
dimension $M=\left[\frac{3n+1}{2}\right]$ is minimal for every even $n$, while for odd $n$
there could be two possible values. (See also \cite[Proposition 9.3.3]{Forstneric2017E}.) 
In 2012, Ho et al.\ \cite{HoJacobowitzLandweber2012} 
found new examples showing that these dimensions are optimal 
already for Grauert tubes around compact totally real submanifolds,
except perhaps for immersions with odd $n$.
A more complete discussion of this topic and a self-contained proof  of Theorem  \ref{th:EGS}
can be found in \cite[Sects.\ 9.2--9.5]{Forstneric2017E}. 
Here we only give a brief outline of the main ideas used in the proof.

One begins by choosing a sufficiently generic almost proper  map $h\colon X\to \C^{n}$ 
(see Theorem \ref{th:classical}(d)) and then tries to find
the smallest possible number of functions $g_1,\ldots, g_q\in \Ocal(X)$ such that the map
\begin{equation}\label{eq:f}
	f=(h,g_1,\ldots,g_q)\colon X\to \C^{n+q}
\end{equation} 
is a proper embedding or immersion. Starting with a big number of functions 
$\tilde g_1,\ldots, \tilde g_{\tilde q}\in \Ocal(X)$ which do the job, we try to reduce their number 
by applying a suitable fibrewise linear projection 
onto a smaller dimensional subspace, where the projection  depends holomorphically on the base point.  
Explicitly, we look for functions 
\begin{equation}\label{eq:ajk}
	g_j=\sum_{k=1}^{\tilde q} a_{j,k}\tilde g_k, \qquad a_{j,k}\in \Ocal(X),\ \ j=1,\ldots,q
\end{equation} 
such that the map \eqref{eq:f} is a proper embedding or immersion.
In order to separate those pairs of points in $X$ which are not separated by the base map $h\colon X\to \C^{n}$,
we consider coefficient functions of the form  $a_{j,k}=b_{j,k}\circ h$ where $b_{j,k}\in\Ocal(\C^n)$
and $h\colon X\to\C^n$ is the chosen base almost proper map.

This outline cannot be applied directly since the base map  $h\colon X\to\C^n$ may have too complicated behavior.
Instead, one proceeds by induction on strata in a suitably chosen complex analytic stratification of $X$ which is 
equisingular with respect to $h$. The induction steps are of two kinds. 
In  a step of the first kind we find a map $g=(g_1,\ldots,g_q)$ \eqref{eq:ajk} which separates 
points on the (finite) fibres of $h$ over the next bigger stratum and matches the map from the previous step
on the union of the previous strata (the latter is a closed complex subvariety of $X$). A step of the second kind
amounts to removing the kernel of the differential $dh_x$ for all points in the next stratum, thereby ensuring 
that $df_x=dh_x\oplus dg_x$ is injective there. 
Analysis of the immersion condition shows that the graph of the map $\alpha=(a_{j,k}) \colon X\to \C^{q{\tilde q}}$ 
in \eqref{eq:ajk} over the given stratum must avoid a certain complex subvariety $\Sigma$ of $E=X\times \C^{q{\tilde q}}$ 
with algebraic fibres. Similarly, analysis of the point separation condition leads to the problem of finding a
map $\beta=(b_{j,k})\colon \C^n\to \C^{q{\tilde q}}$ avoiding a certain complex subvariety 
of $E=\C^n\times \C^{q{\tilde q}}$ with algebraic fibers.
In both cases the projection $\pi\colon E\setminus \Sigma\to X$ is a stratified holomorphic fibre bundle
all of whose fibres are Oka manifolds. More precisely, if $q\ge \left[\frac{n}{2}\right]+1$ then
each fibre $\Sigma_x= \Sigma\cap\, E_x$ is either empty or a union of finitely many affine linear subspaces of 
$E_x$ of complex codimension $>1$. The same lower bound on $q$ guarantees the existence of a continuous section 
$\alpha\colon X\to E\setminus \Sigma$  avoiding $\Sigma$. 
Gromov's  Oka principle \cite{Gromov1989} then furnishes a holomorphic section $X\to E\setminus \Sigma$.
A general Oka principle for sections of stratified holomorphic fiber bundles 
with Oka fibres is given by \cite[Theorem 5.4.4]{Forstneric2017E}.
We refer the reader to the original papers or to \cite[\S 9.3--9.4]{Forstneric2017E}
for further details.

The classical constructions of proper holomorphic embeddings 
of Stein manifolds into Euclidean spaces are coordinate dependent and hence 
do not generalize to more general target manifolds.
A conceptually new method has been found recently by 
Ritter and the author \cite{ForstnericRitter2014MZ}. It is based on a method of separating
certain pairs of compact polynomially convex sets in $\C^N$ by Fatou-Bieberbach domains
which contain one of the sets and avoid the other one.
Another recently developed method, which also depends on holomorphic automorphisms and applies
to a much bigger class of target manifolds, is discussed in \S \ref{sec:density}.

%
%
\subsection{Embedding open Riemann surfaces into $\C^2$}
\label{ss:RS}
The constructions described so far fail to embed open Riemann surfaces into $\C^2$. 
The problem is that the subvarieties $\Sigma$ in the proof of Theorem \ref{th:EGS} 
may contain hypersurfaces, and hence the Oka principle for sections of $E\setminus \Sigma\to X$
fails in general due to hyperbolicity of its complement. 
It is still an open problem whether every open Riemann surface embeds as a smooth
closed complex curve in $\C^2$. (By Theorem \ref{th:classical} it embeds 
properly holomorphically into $\C^3$ and immerses with normal crossings into $\C^2$. 
Every compact Riemann surface embeds
holomorphically into $\CP^3$ and immerses into $\CP^2$, but very few of them embed into $\CP^2$;
see \cite{GriffithsHarris1994}.)   There are no topological obstructions to this problem ---
it was shown by Alarc\'on and L{\'o}pez \cite{AlarconLopez2013} that every open orientable surface 
$S$ carries a complex structure $J$ such that the Riemann surface $X=(S,J)$ 
admits a proper holomorphic embedding into $\C^2$.

There is a variety of results in the literature concerning the existence of proper holomorphic embeddings
of certain special open Riemann surfaces into $\C^2$; the reader may wish to consult the survey 
in \cite[\S 9.10--9.11]{Forstneric2017E}. Here we mention only a few of the currently most
general known results on the subject. The first one from 2009, due to Wold and the author, concerns 
bordered Riemann surfaces.

%
%
%
%
\begin{theorem} \label{th:FW1}
{\rm \cite[Corollary 1.2]{ForstnericWold2009}}
Assume that $X$ is a compact bordered Riemann surface with boundary of class $\Cscr^r$ for some $r>1$. 
If $f \colon X \hra\C^2$ is a $\Cscr^1$ embedding that is holomorphic in the interior 
$\mathring X=X\setminus bX$, then $f$ can be approximated uniformly on compacts in 
$\mathring X$ by proper holomorphic embeddings $\mathring X\hra\C^2$.          
\end{theorem}

The proof relies on techniques introduced mainly by Wold \cite{Wold2006,Wold2006-2,Wold2007}.
One of them concerns exposing boundary points of an embedded bordered Riemann surface in $\C^2$.
This technique was improved in \cite{ForstnericWold2009}; see also the exposition in \cite[\S 9.9]{Forstneric2017E}.
The second one depends on methods of Anders\'en-Lempert theory concerning holomorphic automorphisms
of Euclidean spaces; see \S\ref{ss:wild}.
A proper holomorphic embedding $\mathring X \hra \C^2$ is obtained by
first exposing a boundary point in each of the boundary curves of $f(X)\subset \C^2$, sending these
points to infinity by a rational shear on $\C^2$ without other poles on $f(X)$, 
and then using a carefully constructed sequence of holomorphic automorphisms of $\C^2$
whose domain of convergence is a Fatou-Bieberbach domain $\Omega\subset \C^2$ 
which contains the embedded complex curve $f(X)\subset \C^2$, but does not contain any of its boundary points.
If $\phi\colon \Omega\to \C^2$ is a Fatou-Bieberbach map then $\phi\circ f \colon X\hra\C^2$
is a proper holomorphic embedding.
A complete exposition of this proof can also be found in \cite[\S 9.10]{Forstneric2017E}.

The second result due to Wold and the author \cite{ForstnericWold2013} (2013)
concerns domains with infinitely many boundary components.
A domain $X$ in the Riemann sphere $\mathbb P^1$ is a \emph{generalized circled domain} 
if every connected component of $\P^1 \setminus X$ is a round disc or a point.
Note that $\P^1 \setminus X$ contains at most countably many discs. 
By  the uniformization theorem of He and Schramm  \cite{HeSchramm1993,HeSchramm1995}, 
every domain in $\mathbb P^1$ with at most countably many complementary components
is conformally equivalent to a generalized circled domain.

%
%
%
%
\begin{theorem} \label{th:FW2}
{\rm \cite[Theorem 5.1]{ForstnericWold2013}}
Let $X$ be a generalized circled domain in $\P^1$. 
If all but finitely many punctures in $\mathbb P^1\setminus X$ are limit points of discs in $\mathbb P^1\setminus X$, 
then $X$ embeds properly holomorphically into $\mathbb C^2$.
\end{theorem}

The paper \cite{ForstnericWold2013} contains several other more precise results on this subject.

The special case of Theorem \ref{th:FW2} for plane domains $X\subset \C$ bounded by finitely many 
Jordan curves (and without punctures) is due to Globevnik and Stens{\o}nes \cite{GlobevnikStensones1995}. Results 
on embedding certain Riemann surfaces with countably many boundary components into $\C^2$ were also proved
by Majcen \cite{Majcen2009}; an exposition can be found in \cite[\S 9.11]{Forstneric2017E}.
The proof of Theorem \ref{th:FW2} relies on similar techniques as that of  Theorem \ref{th:FW1}, 
but it uses a considerably more involved induction scheme for dealing  with infinitely many boundary components, 
clustering them together into suitable subsets to which the available analytic methods can be applied.  
The same technique gives the analogous result for domains in tori.

There are a few other recent results concerning embeddings of open Riemann surfaces
into $\C\times \C^*$ and $(\C^*)^2$, where $\C^*=\C\setminus \{0\}$.
Ritter showed in \cite{Ritter2013JGEA} that, for every circular domain $X\subset \D$ 
with finitely many boundary components, 
each homotopy class of continuous maps $X\to \C\times \C^*$ 
contains a proper holomorphic map. If $\D\setminus X$ contains
finitely many punctures, then every continuous map $X\to \C\times \C^*$ is homotopic
to a proper holomorphic immersion that identifies at most finitely many pairs of points in $X$
(L{\'arusson and Ritter  \cite{LarussonRitter2014}). Ritter \cite{Ritter2014} also gave an
analogue of Theorem} \ref{th:FW1} for proper holomorphic embeddings 
of certain open Riemann surfaces into $(\C^*)^2$.

%
%
\subsection{Automorphisms of Euclidean spaces and wild embeddings}
\label{ss:wild}
There is another line of investigations  that we wish to touch upon. It concerns the question
how many proper holomorphic embeddings $X\hra \C^N$ of a given Stein manifold $X$ 
are there up to automorphisms of $\C^N$, and possibly also of $X$.
This question was motivated  by certain famous results from algebraic geometry, 
such as the one of Abhyankar and Moh \cite{AbhyankarMoh1975} and Suzuki \cite{Suzuki1974}
to the effect that every polynomial embedding $\C\hra\C^2$ is equivalent to the
linear embedding $z\mapsto (z,0)$ by a polynomial automorphism of $\C^2$.

It is a basic fact that for any $N>1$ the holomorphic automorphism group $\Aut(\C^N)$ is very big 
and complicated. This is in stark contrast to the situation for bounded 
or, more generally, hyperbolic domains in $\C^N$ which have few automorphism;
see Greene et al. \cite{Greene2011} for a survey of the latter topic. 
 Major early work on understanding the group $\Aut(\C^N)$ was made
by Rosay and Rudin \cite{RosayRudin1988}. This theory became very 
useful with the papers of Anders{\'e}n and Lempert \cite{AndersenLempert1992} 
and Rosay and the author \cite{ForstnericRosay1993} in 1992--93. The central result
is that every map in a smooth isotopy of biholomorphic mappings
$\Phi_t\colon \Omega=\Omega_0 \to \Omega_t$ $(t\in [0,1])$ 
between Runge domains in $\C^N$, with $\Phi_0$ the identity on $\Omega$,
can be approximated uniformly on compacts in $\Omega$ by holomorphic automorphisms of $\C^N$
(see \cite[Theorem 1.1]{ForstnericRosay1993} or \cite[Theorem 4.9.2]{Forstneric2017E}).
The analogous result holds on any Stein manifold with the density property;
see \S\ref{sec:density}. A comprehensive survey of this subject can be found in  
\cite[Chap.\ 4]{Forstneric2017E}.

By twisting a given submanifold of $\C^N$ with a sequence of holomorphic automorphisms, 
one can show that for any pair of integers $1\le n<N$
the set of all equivalence classes of proper holomorphic embeddings 
$\C^n\hra \C^N$, modulo automorphisms of both spaces, 
is uncountable (see \cite{DerksenKutzschebauchWinkelmann1999}). 
In particular,  the Abhyankar-Moh theorem fails 
in the holomorphic category since there exist proper holomorphic
embeddings $\phi \colon \C\hra\C^2$ that are nonstraightenable by automorphisms 
of $\C^2$ \cite{ForstnericGlobevnikRosay1996}, as well as embeddings  whose
complement $\C^2\setminus \phi(\C)$ is Kobayashi hyperbolic \cite{BuzzardFornaess1996}.
More generally, for any pair of integers $1\le n<N$ there exists a proper holomorphic embedding 
$\phi\colon \C^n\hra \C^N$ such that every nondegenerate holomorphic map 
$\C^{N-n}\to \C^N$ intersects $\phi(\C^n)$ at infinitely many points \cite{Forstneric1999JGA}.
It is also possible to arrange that $\C^N\setminus \phi(\C^n)$ 
is Eisenman $(N-n)$-hyperbolic \cite{BorellKutzschebauch2006}. 
A more comprehensive discussion of this subject can be found in 
\cite[\S 4.18]{Forstneric2017E}.

By using nonlinearizable proper holomorphic embeddings $\C\hra\C^2$,
Derksen and Kutzschauch gave the first known examples of nonlinearizable
periodic automorphisms of $\C^n$ \cite{DerksenKutzschebauch1998}.
For instance, there is a nonlinearizable holomorphic involution on $\C^4$. 

In another direction, Baader et al.\ \cite{Baaderall2010} constructed an example of a properly 
embedded disc in $\C^2$ whose image is topologically knotted, thereby answering
a questions of Kirby. It is unknown whether there exists a knotted proper
holomorphic embedding $\C\hra\C^2$, or an unknotted proper
holomorphic embedding $\D\hra \C^2$ of the disc.

Automorphisms of $\C^2$ and $\C^*\times \C$ were used in a very clever way by Wold
in his landmark construction of non-Runge Fatou-Bieberbach domains in $\C^2$ 
\cite{Wold2008} and of non-Stein long $\C^2$'s \cite{Wold2010}. Each of these results 
solved a long-standing open problem.
More recently, Wold's construction was developed further 
by Boc Thaler and the author \cite{BocThalerForstneric2016} who showed 
that there is a continuum of pairwise nonequivalent long $\C^n$'s for any $n>1$ which do 
not admit any nonconstant holomorphic or plurisubharmonic functions.
(See aso \cite[4.21]{Forstneric2017E}.)


\section{Embeddings into Stein manifolds with the density property}
\label{sec:density}

\subsection{Universal Stein manifolds}\label{ss:universal}
It is natural to ask which Stein manifolds, besides the Euclidean spaces, 
contain all Stein manifolds of suitably low dimension as closed complex submanifolds.
To facilitate the discussion, we introduce the following notions.

\begin{definition}\label{def:universal}
Let $Y$ be a Stein manifold.
\begin{enumerate} 
\item 
$Y$ is {\em universal for proper holomorphic embeddings} if every Stein manifold $X$ with $2\dim X<\dim Y$ 
admits a proper holomorphic embedding $X\hra Y$. 
\vspace{1mm}
\item
$Y$ is {\em strongly universal for proper holomorphic embeddings} if, under the assumptions in (1), 
every continuous map $f_0\colon X\to Y$ which is holomorphic in a neighborhood of a compact 
$\Ocal(X)$-convex set $K\subset X$ is homotopic to a proper holomorphic embedding $f_0\colon X\hra Y$ 
by a homotopy $f_t\colon X\to Y$ $(t\in[0,1])$ such that $f_t$ is holomorphic and arbitrarily close to $f_0$ 
on $K$ for every $t\in [0,1]$.
\vspace{1mm}
\item 
$Y$ is (strongly) {\em universal for proper holomorphic immersions} if  condition (1) (resp.\ (2))
holds for proper holomorphic immersions $X\to Y$ from any Stein manifold $X$ satisfying $2\dim X\le \dim Y$. 
\end{enumerate}
\end{definition}

In the terminology of Oka theory (cf.\ \cite[Chap.\ 5]{Forstneric2017E}), 
a complex manifold $Y$ is (strongly) universal for proper holomorphic
embeddings if it satisfies the basic Oka property (with approximation) for proper holomorphic
embeddings $X\to Y$ from Stein manifolds of dimension $2\dim X<\dim Y$. 
The dimension hypotheses in the above definition are justified by Proposition \ref{prop:generic}.
The main goal is to find good sufficient conditions for a Stein manifold to be universal.
If a manifold $Y$ is Brody hyperbolic \cite{Brody1978} (i.e., it does not admit any nonconstant holomorphic images of $\C$)
then clearly no complex manifold containing a nontrivial holomorphic image of $\C$ can be embedded into $Y$. 
In order to get positive results, one must therefore assume that $Y$ enjoys a suitable
holomorphic flexibility (anti-hyperbolicity) property.

\begin{problem}\label{prob:Oka}
Is every Stein Oka manifold (strongly) universal for proper holomorphic embeddings and immersions?
\end{problem}

Recall (see e.g.\ \cite[Theorem 5.5.1]{Forstneric2017E}) that every Oka manifold is strongly universal 
for not necessarily proper holomorphic maps, embeddings and immersions. Indeed, the cited theorem asserts 
that a generic holomorphic map $X\to Y$ from a 
Stein manifold $X$ into an Oka manifold $Y$ is an immersion if $\dim Y\ge 2\dim X$, and is an injective
immersion if $\dim Y > 2\dim X$. However, the Oka condition does not imply universality for {\em proper}
holomorphic maps since there are examples of (compact or noncompact) Oka manifolds 
without any closed complex subvarieties of positive dimension (see \cite[Example 9.8.3]{Forstneric2017E}).  

\subsection{Manifolds with the (volume) density property}
The following condition was introduced in 2000 by Varolin \cite{Varolin2000,Varolin2001}.

\begin{definition}
\label{def:density}
A complex manifold $Y$ enjoys the (holomorphic) {\em density property}
if the Lie algebra generated by the $\C$-complete holomorphic vector fields on $Y$
is dense in the Lie algebra of all holomorphic vector fields in the compact-open topology.

A complex manifold $Y$ endowed with a holomorphic volume form $\omega$ enjoys the {\em volume density property} 
if the analogous density condition holds in the Lie algebra of all holomorphic vector fields
on $Y$ with vanishing $\omega$-divergence.
\end{definition}

The algebraic density and volume density properties were introduced by Kaliman and Kutzsche\-bauch
\cite{KalimanKutzschebauch2010IM}.  The class of Stein manifolds with the (volume) density property
includes most complex Lie groups and homogeneous spaces, as well as many nonhomogeneous manifolds. 
We refer to \cite[\S 4.10]{Forstneric2017E} for a more complete 
discussion and an up-to-date collection of references on this subject. 
Another recent survey is the paper by Kaliman and Kutzschebauch \cite{KalimanKutzschebauch2015}.
Every complex manifold with the density property is an Oka manifold, and
a Stein manifold with the density property is elliptic in the sense of Gromov
(see \cite[Proposition 5.6.23]{Forstneric2017E}).
It is an open problem whether a contractible Stein manifold with the 
density property is biholomorphic to a complex Euclidean space.

The following result is due to Andrist and Wold \cite{AndristWold2014} in the special case
when $X$ is an open Riemann surface, to Andrist et al. \cite[Theorems 1.1, 1.2]{AndristFRW2016}
for embeddings, and to the author \cite[Theorem 1.1]{Forstneric-immersions} for immersions
in the double dimension. 

%
%
\begin{theorem} \label{th:density}
{\rm \cite{AndristFRW2016,AndristWold2014,Forstneric-immersions}}
Every Stein manifold with the density or the volume density property is strongly universal
for proper holomorphic embeddings and immersions.
\end{theorem}

To prove Theorem \ref{th:density}, one follows the scheme of proof of the Oka principle for maps
from Stein manifolds to Oka manifolds (see \cite[Chapter 5]{Forstneric2017E}), but with a crucial
addition which we now briefly describe. 

Assume that $D\Subset X$ is a relatively compact strongly pseudoconvex 
domain with smooth boundary and $f\colon \overline D\hra Y$ is a holomorphic embedding
such that $f(bD) \subset Y\setminus L$, where $L$ is a given compact $\Ocal(Y)$-convex set 
in $Y$. We wish to approximate $f$ uniformly on $\overline D$ by a holomorphic embedding 
$f'\colon \overline{D'}\hra Y$ of a certain bigger strongly pseudoconvex domain $\overline{D'} \Subset X$ to $Y$, 
where $D'$ is either a union of $D$ with a small convex bump $B$
chosen such that $f(\overline {D\cap B})\subset Y\setminus L$, or a thin handlebody
whose core is the union of $D$ and a suitable smoothly embedded totally real disc in $X\setminus D$.
(The second case amounts to a change of topology of the domain, and it typically occurs
when passing a critical point of a strongly plurisubharmonic exhaustion function on $X$.)
In view of Proposition \ref{prop:generic}, we only need to approximate $f$ by a holomorphic map
$f'\colon \overline{D'} \to Y$ since a small generic perturbation of $f'$ then yields an embedding. 
It turns out that the second case involving a handlebody 
easily reduces to the first one by applying a Mergelyan type approximation theorem; 
see \cite[\S 5.11]{Forstneric2017E} for this reduction. 
The attachment of a bump is handled by using the density property 
of $Y$. This property allows us to find a holomorphic map $g\colon \overline B\to Y\setminus L$ 
approximating $f$ as closely as desired on a neighborhood of the attaching set $\overline{B\cap D}$
and satisfying $g(\overline B)\subset Y\setminus L$.
(More precisely, we use that isotopies of biholomorphic maps between
pseudoconvex Runge domains in $Y$ can be approximated by holomorphic automorphisms
of $Y$; see  \cite[Theorem 1.1]{ForstnericRosay1993} 
and also \cite[Theorem 4.10.5]{Forstneric2017E} for the version pertaining  to Stein manifolds with the
density property.) Assuming that $g$ is sufficiently close to $f$ on $\overline{B\cap D}$, 
we can glue them into a holomorphic map $f'\colon \overline{D'}\to Y$ which approximates $f$ on $\overline D$ 
and satisfies $f'(\overline{B})\subset Y\setminus L$.
The proof is completed by an induction procedure in which every induction 
step is of the type described above. The inclusion $f'(\overline{B})\subset Y\setminus L$
satisfied by the next map in the induction step guarantees properness of the limit embedding $X\hra Y$.
Of course the sets $L\subset Y$ also increase and form an exhaustion of $Y$.

The case of immersions in double dimension requires a more precise analysis. 
In the induction step described above, we must ensure that the immersion 
$f \colon \overline D\to Y$ is injective (an embedding) on the attaching set $\overline{B\cap D}$
of the bump $B$. This can be arranged by general position provided that $\overline{B\cap D}$
is very thin. It is shown in \cite{Forstneric-immersions} that it suffices to work with convex
bumps such that, in suitably chosen holomorphic coordinates on a neighborhood of $\overline B$,
the set $B$ is a convex polyhedron and $\overline{B\cap D}$ is a very thin neighborhood 
of one of its faces. This means that $\overline{B\cap D}$ is small thickening of a
$(2n-1)$-dimensional object in $X$, and hence we can easily arrange that $f$ 
is injective  on it. The remainder of the proof proceeds exactly as before,
completing our sketch of proof of Theorem \ref{th:density}.

%
%
\subsection{On the Schoen-Yau conjecture}\label{ss:S-Y}
The following corollary to Theorem \ref{th:density} is related to a conjecture of  Schoen and Yau \cite{SchoenYau1997}
that the disc $\D=\{\zeta \in\C:|\zeta|<1\}$ does not admit any proper harmonic maps to $\R^2$.

\begin{corollary}\label{cor:harmonic}
Every Stein manifold $X$ of complex dimension $n$ admits a proper holomorphic immersion to $(\C^*)^{2n}$,
and a proper pluriharmonic map into $\R^{2n}$.
\end{corollary}

\begin{proof}
The space $(\C^*)^n$ with coordinates $z=(z_1,\ldots,z_n)$ (where $z_j\in \C^*$ for $j=1,\ldots,n$)
enjoys the volume density property with respect to the volume form 
\[
	\omega= \frac{dz_1\wedge\cdots\wedge dz_n}{z_1\cdots z_n}.
\]
(See Varolin \cite{Varolin2001} or \cite[Theorem 4.10.9(c)]{Forstneric2017E}.) 
Hence, \cite[Theorem 1.2]{Forstneric-immersions} (the part of
Theorem \ref{th:density} above concerning immersions into the double dimension) furnishes
a proper holomorphic immersion $f=(f_1,\ldots,f_{2n})\colon X\to (\C^*)^{2n}$. 
It follows that the map 
\begin{equation}\label{eq:log}
	u=(u_1,\ldots,u_{2n})\colon X\to \R^{2n}\quad \text{with}\ \  u_j=\log|f_j|\ \  \text{for}\ \ j=1,\ldots, 2n
\end{equation}
is a proper map of $X$ to $\R^{2n}$ whose components are pluriharmonic functions.
\end{proof}

Corollary \ref{cor:harmonic} gives a counterexample to the Schoen-Yau conjecture in 
every dimension and for any  Stein  source manifold. The first and very explicit counterexample 
was given by Bo{\v z}in \cite{Bozin1999IMRN} in 1999.
In 2001, Globevnik and the author \cite{ForstnericGlobevnik2001MRL} constructed 
a proper holomorphic map $f=(f_1,f_2)\colon\D\to\C^2$ whose image is contained in $(\C^*)^2$, i.e.,
it avoids both coordinate axes.
The associated harmonic map $u=(u_1,u_2)\colon \D\to\R^2$ \eqref{eq:log} then satisfies 
$\lim_{|\zeta|\to 1} \max\{u_1(\zeta),u_2(\zeta)\} = +\infty$  which implies properness. 
Next, Alarc{\'o}n and L{\'o}pez \cite{AlarconLopez2012JDG} showed in 2012 
that every open Riemann surface $X$ admits a conformal minimal immersion 
$u=(u_1,u_2,u_3)\colon X\to\R^3$ with a proper (harmonic) projection 
$(u_1,u_2)\colon X\to \R^2$. In 2014, Andrist and Wold \cite[Theorem 5.6]{AndristWold2014}
proved Corollary \ref{cor:harmonic} in the case $n=1$.

Comparing Corollary \ref{cor:harmonic} with the above mentioned result of Globevnik and the 
author \cite{ForstnericGlobevnik2001MRL}, one is led to the following question.

\begin{problem}
Let $X$ be a Stein manifold of dimension $n>1$. Does there exist a proper holomorphic immersion
$f\colon X\to \C^{2n}$ such that $f(X)\subset (\C^*)^{2n}$?
\end{problem}

More generally, one can ask which type of sets in Stein manifolds can be avoided by 
proper holomorphic maps from Stein manifolds of sufficiently low dimension. In this direction, Drinovec Drnov\v sek
showed in \cite{Drinovec2004MRL} that any closed complete pluripolar set can be avoided by 
proper holomorphic discs; see also Borell et al.\ \cite{Borell2008MRL} for embedded discs 
in $\C^n$. Note that every closed complex subvariety is a complete pluripolar set.


\section{Embeddings of strongly pseudoconvex Stein domains}
\label{sec:PSC}

\subsection{The Oka principle for embeddings of strongly pseudoconvex domains}
What can be said about proper holomorphic
embeddings and immersions of Stein manifolds $X$ into arbitrary (Stein) manifolds $Y$?
If $Y$ is Brody hyperbolic \cite{Brody1978}, 
then no complex manifold containing a nontrivial holomorphic image of $\C$ 
embeds into $Y$. However, if $\dim Y>1$ and $Y$ is Stein then $Y$ still admits proper holomorphic images of 
any bordered Riemann surface \cite{DrinovecForstneric2007DMJ,Globevnik2000}.
For domains in Euclidean spaces, this line of investigation was 
started in 1976 by Forn{\ae}ss \cite{Fornaess1976} and continued in 1985 by 
L{\o}w \cite{Low1985MZ} and the author \cite{Forstneric1986TAMS} who proved
that every bounded strongly pseudoconvex domain $X\subset \C^n$ admits
a proper holomorphic embedding into a high dimensional polydisc and ball. 
The long line of subsequent developments culminated in the following 
result of Drinovec Drnov\v sek and the author \cite{DrinovecForstneric2007DMJ,DrinovecForstneric2010AJM}.

%
%
%
%
\begin{theorem}
\label{th:BDF2010}
{\rm \cite[Corollary 1.2]{DrinovecForstneric2010AJM}}
Let $X$ be a relatively compact, smoothly bounded, strongly pseudo\-convex domain in 
a Stein manifold $\wt X$ of dimension $n$, and let $Y$ be a Stein manifold of dimension $N$. 
If $N>2n$ then every continuous map $f\colon \overline X \to Y$ which is holomorphic on $X$
can be approximated uniformly on compacts in $X$ by proper holomorphic embeddings $X\hra Y$.
If $N\ge 2n$ then the analogous result holds for immersions.
The same conclusions hold if the manifold $Y$ is strongly $q$-complete for some $q\in \{1,2,\ldots, N-2n+1\}$,
where the case $q=1$ corresponds to Stein manifolds.
\end{theorem}

In the special case when $Y$ is a domain in a Euclidean space, this  is
due to Dor \cite{Dor1995}. The papers \cite{DrinovecForstneric2007DMJ,DrinovecForstneric2010AJM}
include several more precise results in this direction and references to numerous previous works. 
Note that a continuous map $f\colon  \overline X \to Y$ from a compact strongly pseudoconvex domain
which is holomorphic on the open domain $X$, with values in an arbitrary complex manifold $Y$, can be 
approximated uniformly on $\overline X$ by holomorphic maps from small open neighborhoods of $\overline X$ in 
the ambient manifold $\wt X$, where the neighborhood depends on the map (see 
\cite[Theorem 1.2]{DrinovecForstneric2008FM} or \cite[Theorem 8.11.4]{Forstneric2017E}). 
However, unless $Y$ is an Oka manifold, it is impossible to approximate $f$ uniformly on $\overline X$ by 
holomorphic maps from a fixed bigger domain $X_1\subset \wt  X$ independent of the map. 
For this reason, it is imperative that the initial map $f$ in Theorem \ref{th:BDF2010} 
be defined on all of $\overline X$.

One of the main techniques used in the proof of Theorem \ref{th:BDF2010} are 
special holomorphic peaking functions on $X$. The second tool 
is the method of holomorphic sprays developed in the context of Oka theory; this is essentially
a nonlinear version of the $\dibar$-method. 

Here is the main idea of the proof of Theorem \ref{th:BDF2010}. 
Choose a strongly $q$-convex Morse exhaustion function $\rho\colon Y\to\R_+$.
(When $q=1$, $\rho$ is strongly plurisubharmonic.)
By using the mentioned tools, one can approximate any given holomorphic 
map $f\colon \overline X\to Y$ uniformly on compacts in $X$ by another holomorphic map $\tilde f\colon \overline X\to Y$
such that $\rho\circ \tilde f > \rho\circ f + c$ holds on $bX$ for some constant $c>0$ depending only on 
the geometry of $\rho$ on a given compact set $L\subset Y$ containing $f(\overline X)$.
Geometrically speaking, this means that we lift the image of the boundary of $X$  in $Y$ 
to a higher level of the function $\rho$ by a prescribed amount.
At the same time, we can ensure that $\rho\circ \tilde f > \rho \circ  f -\delta$ on $X$ 
for any given $\delta>0$, and that $\tilde f$ approximates $f$ as closely as desired on a given 
compact $\Ocal(X)$-convex subset $K\subset X$. 
By Proposition \ref{prop:generic} we can ensure that our maps are embeddings. 
An inductive application of this technique yields
a sequence of holomorphic embeddings $f_k\colon\overline X\hra Y$
converging to a proper holomorphic embedding $X\hra Y$.
The same construction gives proper holomorphic immersions when $N\ge 2n$.

%
%
\subsection{On the Hodge Conjecture for $q$-complete manifolds}\label{ss:Hodge}
A more precise analysis of the proof of Theorem \ref{th:BDF2010} 
was used by Smrekar, Sukhov and the author  \cite{FSS2016}
to show the following result along the lines of the Hodge conjecture.

\begin{theorem}
If $Y$ is a $q$-complete complex manifold of dimension $N$ and of finite topology 
such that $q<N$ and the number $N+q-1=2p$ is even, then every cohomology 
class in $H^{N+q-1}(Y;\Z)$ is Poincar\'e dual to an analytic cycle in $Y$ consisting of 
proper holomorphic images of the ball  $\B^p\subset \C^p$.
If the manifold $Y$ has infinite topology, the same result holds for elements of the group
$\Hscr^{N+q-1}(Y;\Z) = \lim_j H^{N+q-1}(M_j;Z)$ where $\{M_j\}_{j\in \N}$ is an exhaustion of $Y$ 
by compact smoothly bounded domains.
\end{theorem}

Note that  $H^{N+q-1}(Y;\Z)$ is the highest dimensional a priori nontrivial
cohomology group of a $q$-complete manifold $Y$ of dimension $N$.
We do not know  whether a similar result holds for lower dimensional cohomology groups
of a $q$-complete manifold. In the special case when $Y$ is a Stein manifold,
the situation is better understood thanks to the Oka-Grauert principle, 
and the reader can find appropriate references in the paper \cite{FSS2016}.

%
%
\subsection{Complete bounded complex submanifolds}\label{ss:complete}
There are interesting recent constructions of properly embedded complex submanifolds
$X\subset \B^N$ of the unit ball in $\C^N$ (or of pseudoconvex domains in $\C^N$) which are {\em complete}
in the sense that every curve in $X$ terminating on the sphere $b\B^N$ has infinite length.
Equivalently, the metric on $X$, induced from the Euclidean metric on $\C^N$ by the embedding
$X\hra \C^N$, is a complete metric.

The question whether there exist complete bounded complex submanifolds 
in Euclidean spaces was asked by Paul Yang in 1977.
The first such examples were provided by Jones \cite{Jones1979PAMS} in 1979.
Recent results on this subject are due to
Alarc{\'o}n  and the author \cite{AlarconForstneric2013MA}, Alarc{\'o}n and L{\'opez} \cite{AlarconLopez2016}, 
Drinovec Drnov{\v s}ek \cite{Drinovec2015JMAA}, Globevnik \cite{Globevnik2015AM,Globevnik2016JMAA,Globevnik2016MA},  
and Alarc{\'o}n et al.\ \cite{AlarconGlobevnik2017,AlarconGlobevnikLopez2016Crelle}. 
In \cite{AlarconForstneric2013MA} it was shown that any bordered Riemann surface 
admits a proper complete holomorphic immersion into $\B^2$ and embedding into $\B^3$ 
(no change of the complex structure on the surface is necessary). 
In \cite{AlarconGlobevnik2017} the authors showed that properly embedded complete 
complex curves in the ball $\B^2$ can have any topology, but their method (using holomorphic
automorphisms) does not allow one to control the complex structure of the examples.
Drinovec Drnov{\v s}ek \cite{Drinovec2015JMAA} proved that every strongly 
pseudoconvex domain embeds as a complete complex submanifold of a high 
dimensional ball. Globevnik proved \cite{Globevnik2015AM,Globevnik2016MA} that 
any pseudoconvex domain in $\C^N$ for $N>1$ can be foliated by complete complex hypersurfaces 
given as level sets of a holomorphic function, and Alarc{\'o}n showed \cite{Alarcon2018}
that there are nonsingular foliations of this type goven as level sets of a holomorphic function
without critical points. Furthermore, there is a complete proper holomorphic embedding
$\D\hra\B^2$ whose image contains any given discrete subset of $\B^2$  \cite{Globevnik2016JMAA},
and there exist complex curves of arbitrary topology in $\B^2$ satisfying this property \cite{AlarconGlobevnik2017}.
The constructions in these papers, except those in \cite{Alarcon2018,Globevnik2015AM,Globevnik2016MA}, 
rely on one of the following two methods: 
\begin{itemize}
\item[\rm (a)] Riemann-Hilbert boundary values problem (or holomorphic peaking functions in the
case of higher dimensional domains considered in \cite{Drinovec2015JMAA}); 
\vspace{1mm}
\item[\rm (b)] holomorphic automorphisms of the ambient space $\C^N$. 
\end{itemize}
Each of these methods can be used to increase the intrinsic 
boundary distance in an embedded or immersed submanifold. The first method has the advantage
of preserving the complex structure, and the disadvantage of introducing self-intersections in the double dimension or below. 
The second method is precisely the opposite --- it keeps embeddedness, but does not
provide any control of the complex structure since one must cut away pieces of the image manifold 
to keep it suitably bounded. The first of these methods has recently been applied
in the theory of minimal surfaces in $\R^n$; we refer to the
papers \cite{AlarconDrinovecForstnericLopez2015PLMS,AlarconDrinovecForstnericLopez2017TAMS,
AlarconForstneric2015MA} and the references therein. On the other hand, ambient automorphisms
cannot be applied in minimal surface theory since the only class of self-maps of $\R^n$  $(n>2)$
mapping minimal surfaces to minimal surfaces are the rigid affine linear maps.

Globevnik's method in \cite{Globevnik2015AM,Globevnik2016MA} is 
different from both of the above. He showed that for every integer $N>1$ there is a holomorphic function $f$ on 
the ball $\B^N$ whose real part $\Re f$ is unbounded on every path of finite length that ends on $b\B^N$.
It follows that every level set $M_c=\{f=c\}$ is a closed complete complex hypersurface
in $\B^N$, and $M_c$ is smooth for most values of $c$ in view of Sard's lemma.
The function $f$ is constructed such that its real part grows sufficiently fast on a certain
labyrinth $\Lambda\subset \B^N$, consisting of  pairwise disjoint closed polygonal domains in 
real affine hyperplanes, such that every curve in $\B^N\setminus \Lambda$ which terminates on 
$b\B^N$ has infinite length. The advantage of his method is that it gives an affirmative
answer to Yang's question in all dimensions and codimensions.
The disadvantage is that one cannot control the topology or 
the complex structure of the level sets. By using instead holomorphic automorphisms in 
order to push a submanifold off the labyrinth $\Lambda$, Alarc{\'o}n et al.\ 
\cite{AlarconGlobevnikLopez2016Crelle} succeeded to 
obtain partial control of the topology of the embedded submanifold,
and complete control in the case of complex curves \cite{AlarconGlobevnik2017}.
Finally, by using the method of constructing noncritical holomorphic functions
due to Forstneri\v c \cite{Forstneric2003AM}, Alarc{\'o}n \cite{Alarcon2018} improved Globevnik's main 
result from  \cite{Globevnik2015AM}
by showing that every closed complete complex hypersurface in the ball $\B^n$ $(n>1)$ 
is a leaf in a nonsingular holomorphic foliation of $\B^n$  by closed complete complex hypersurfaces.

By using the labyrinths constructed in \cite{AlarconGlobevnikLopez2016Crelle,Globevnik2015AM}
and methods of Anders\'en-Lempert theory, 
Alarc{\'o}n and the author showed in \cite{AlarconForstneric2018PAMS} that there exists 
a complete injective holomorphic immersion $\mathbb{C}\to\mathbb{C}^2$ whose image is 
everywhere dense in $\mathbb{C}^2$ \cite[Corollary 1.2]{AlarconForstneric2018PAMS}. 
The analogous result holds for any closed complex submanifold 
$X\subsetneqq \mathbb{C}^n$ for $n>1$ (see \cite[Theorem 1.1]{AlarconForstneric2018PAMS}).
Furthermore, if $X$ intersects the ball $\mathbb{B}^n$ and $K$ is a connected 
compact subset of $X\cap\mathbb{B}^n$, then there is a Runge domain $\Omega\subset X$ containing $K$ 
which admits a complete injective holomorphic immersion $\Omega\to\mathbb{B}^n$ whose image is 
dense in $\mathbb{B}^n$.

%
%
\subsection{Submanifolds with exotic boundary behaviour}\label{ss:exotic}
The boundary behavior of proper holomorphic maps between bounded  
domains with smooth boundaries in complex Euclidean spaces has been studied extensively;
see  the recent survey by Pinchuk et al.\ \cite{Pinchuk2017}.
It is generally believed, and has been proved under a variety of additional conditions,
that proper holomorphic maps between relatively compact smoothly bounded domains 
of the same dimension always extend smoothly
up to the boundary. In dimension $1$ this is the classical theorem of Carath{\'e}odory
(see \cite{Caratheodory1913MA} or \cite[Theorem 2.7]{Pommerenke1992}).
On the other hand, proper holomorphic maps into higher dimensional domains
may have rather wild boundary behavior. For example, Globevnik \cite{Globevnik1987MZ} 
proved in 1987 that, given $n\in \N$, if $N\in\N$ is sufficiently large then there 
exists a continuous map $f\colon \overline \B^n \to \overline \B^N$ which is holomorphic in 
$\B^n$ and satisfies $f(b\B^n)=b\B^N$. Recently, the author  \cite{Forstneric2017Sept}  constructed 
a properly embedded holomorphic disc $\D\hra \B^2$ in the $2$-ball 
with arbitrarily small area (hence it is the zero set of a bounded holomorphic function on $\mathbb{B}^2$
according to Berndtsson \cite{Berndtsson1980}) which extends holomorphically across the boundary of the disc,
with the exception of one boundary point, such that its boundary curve is injectively immersed
and everywhere dense in the sphere $b\mathbb{B}^2$. Examples of proper (not necessarily embedded) 
discs with similar behavior were found earlier by Globevnik and Stout \cite{GlobevnikStout1986}.


\section{The soft Oka principle for proper holomorphic embeddings}
\label{sec:soft}

By combining the technique in the proof of Theorem \ref{th:BDF2010} with methods from the papers 
by Slapar and the author \cite{ForstnericSlapar2007MRL,ForstnericSlapar2007MZ} 
one can prove the following seemingly new result.

%
%
\begin{theorem}\label{th:soft}
Let $(X,J)$ and $Y$ be Stein manifolds, where $J\colon TX\to TX$ denotes the complex structure operator 
on $X$.  If $\dim Y > 2\dim X$  then for every continuous map $f\colon X\to Y$ there exists a
Stein structure $J'$ on $X$, homotopic to $J$, and a proper holomorphic embedding $f'\colon (X,J')\hra Y$
homotopic to $f$. If $\dim Y\ge 2\dim X$ then $f'$ can be chosen a proper holomorphic
immersion having only simple double points. The same holds if the manifold $Y$ is $q$-complete 
for some $q\in \{1,2,\ldots, \dim Y-2\dim X+1\}$, where $q=1$ corresponds to Stein manifolds.
\end{theorem}

Intuitively speaking, every Stein manifold $X$ embeds properly holomorphically
into any other Stein manifold $Y$ of dimension $\dim Y >2\dim X$ up to a change of the Stein structure on $X$.
The main result of  \cite{ForstnericSlapar2007MZ} 
amounts to the same statement for holomorphic maps (instead of proper embeddings), but without any 
hypothesis on the target complex manifold $Y$. In order to obtain {\em proper} holomorphic maps $X\to Y$, 
we need a suitable geometric hypothesis on $Y$ in view of the examples of noncompact (even Oka) 
manifolds without any closed complex subvarieties (see \cite[Example 9.8.3]{Forstneric2017E}).

The results from \cite{ForstnericSlapar2007MRL,ForstnericSlapar2007MZ} 
were extended by Prezelj and Slapar \cite{PrezeljSlapar2011} to $1$-convex source manifolds.
For Stein manifolds $X$ of complex dimension $2$, these results 
also stipulate a change of the underlying $\Cscr^\infty$ structure on $X$.
It was later shown by Cieliebak and Eliashberg  that such change is not necessary if one begins with an integrable
Stein structure; see \cite[Theorem 8.43 and Remark 8.44]{CieliebakEliashberg2012}.
For the constructions of exotic Stein structures on smooth orientable $4$-manifolds, in particular
on $\R^4$, see Gompf \cite{Gompf1998,Gompf2005,Gompf2017GT}.

\begin{proof}[Sketch of proof of Theorem \ref{th:soft}]
In order to fully understand the proof, the reader should be familiar with 
\cite[proof of Theorem 1.1]{ForstnericSlapar2007MZ}. (Theorem 1.2 in the same paper gives
an equivalent formulation where one does not change the Stein structure on $X$,
but instead finds a desired holomorphic map on a Stein Runge domain
$\Omega\subset X$ which is diffeotopic to $X$. An exposition  is also available in 
\cite[Theorem 8.43 and Remark 8.44]{CieliebakEliashberg2012} and \cite[\S 10.9]{Forstneric2017E}.) 

We explain the main step in the case $\dim Y > 2\dim X$; the theorem follows by using it inductively
as in \cite{ForstnericSlapar2007MZ}. An interested reader is invited to provide the details.

Assume that $X_0\subset X_1$ is a pair of relatively compact, smoothly 
bounded, strongly pseudoconvex domains in $X$ such that there exists a
strongly plurisubharmonic Morse function $\rho$  on an open set
$U\supset \overline {X_1\setminus X_0}$ in $X$ satisfying
\[
	X_0 \cap U = \{x\in U\colon \rho(x)<a\},\quad 
	 X_1 \cap U = \{x\in U \colon \rho(x)<b\}, 
\]
for a pair of constants $a<b$ and $d\rho\ne 0$ on $bX_0\cup bX_1$.
Let $L_0\subset L_1$ be a pair of compact sets in $Y$.
(In the induction, $L_0$ and $L_1$ are sublevel sets of a strongly $q$-convex exhaustion 
function on $Y$.) Assume that $f_0\colon X\to Y$ is a continuous map whose restriction to a neighborhood of $\overline X_0$ 
is a $J$-holomorphic embedding satisfying $f_0(bX_0) \subset Y\setminus L_0$. 
The goal is to find a new Stein structure $J_1$ on $X$, homotopic to 
$J$ by a smooth homotopy that is fixed in a neighborhood of $\overline X_0$,
such that  $f_0$ can be deformed to a map $f_1\colon X\to Y$ whose restriction 
to a neighborhood of $\overline X_1$ is a $J_1$-holomorphic embedding 
which approximates $f_0$ uniformly on $\overline X_0$ as closely as desired and satisfies
\begin{equation}\label{eq:lifting}
	f_1(\overline {X_1\setminus X_0})\subset Y\setminus L_0, \qquad
	f_1(bX_1) \subset Y\setminus L_1.
\end{equation}
An inductive application of this result proves Theorem \ref{th:soft} as in \cite{ForstnericSlapar2007MZ}. 
(For the case $\dim X=2$, see  \cite[Theorem 8.43 and Remark 8.44]{CieliebakEliashberg2012}.)

By subdividing the problem into finitely many steps of the same kind, it suffices 
to consider the following two basic cases:
\begin{itemize}
\item[\rm (a)] {\em The noncritical case:} $d\rho\ne 0$ on $\overline{X_1\setminus X_0}$.
In this case we say that $X_1$ is a {\em noncritical strongly pseudoconvex extension} of $X_0$.
\vspace{1mm}
\item[\rm (b)] {\em The critical case:} $\rho$ has exactly one critical point $p$ in $\overline{X_1\setminus X_0}$.
\end{itemize}
Let $U_0 \subset U'_0 \subset X$ be a pair of small open neighborhoods of $\overline X_0$ such that 
$f_0$ is an embedding on $U'_0$. Also, let $U_1\subset U'_1\subset X$ 
be small open neighborhoods of $\overline X_1$. 

In case (a), there exists a smooth diffeomorphism $\phi\colon X\to X$ which is 
diffeotopic to the identity map on $X$ by a diffeotopy which is fixed on $U_0\cup (X\setminus U'_1)$
such that $\phi(U_1)\subset U'_0$. The map $\tilde f_0=f_0\circ \phi \colon X \to Y$ is then a holomorphic 
embedding on the set $U_1$ with respect to the Stein structure $J_1=\phi^*(J)$ on $X$ (the pullback of $J$ by $\phi$). 
Applying the lifting procedure in the proof of Theorem \ref{th:BDF2010} and 
up to shrinking $U_1$ around $\overline X_1$, we can homotopically deform $\tilde f_0$  to a continuous
map $f_1\colon X\to Y$ whose restriction to $U_1$ is a $J_1$-holomorphic embedding $U_1\hra Y$ 
satisfying conditions \eqref{eq:lifting}.

In case (b), the change of topology of the sublevel sets of $\rho$ at the critical point $p$
is described by attaching to the strongly pseudoconvex domain $\overline X_0$ a smoothly
embedded totally real  disc $M\subset X_1\setminus X_0$, with $p\in M$ and $bM\subset bX_0$,
whose dimension  equals the Morse index of $\rho$ at $p$.
As shown in \cite{Eliashberg1990,CieliebakEliashberg2012,ForstnericSlapar2007MZ},
$M$ can be chosen such that $\overline X_0\cup M$ has a basis of smooth strongly pseudoconvex
neighborhoods (handlebodies) $H$ which deformation retract onto $\overline X_0\cup M$
such that $X_1$ is a noncritical strongly pseudoconvex extension of $H$. 
Furthermore, as explained in   \cite{ForstnericSlapar2007MZ}, we can homotopically
deform the map $f_0\colon X\to Y$, keeping it fixed in some neighborhood of $\overline X_0$,
to a map that is holomorphic on $H$ and maps $H\setminus \overline X_0$ to 
$L_1\setminus L_0$. By Proposition \ref{prop:generic} we can assume that the new map is a 
holomorphic embedding on $H$. This reduces case (b) to case  (a). 

In the inductive construction, we alternate the application of cases (a) and (b).
If $\dim Y\ge 2\dim X$ then the same procedure applies to immersions.
\end{proof}

A version of this construction, for embedding open Riemann surfaces into $\C^2$ or $(\C^*)^2$
up to a deformation of their complex structure, 
can be found in the papers by Alarc\'on and L{\'o}pez \cite{AlarconLopez2013} and Ritter \cite{Ritter2014}.
However, they use holomorphic automorphisms in order to push the boundary curves
to infinity without introducing self-intersections of the image complex curve. 
The technique in the proof of Theorem \ref{th:BDF2010} will in general introduce self-intersections in double dimension.


\subsection*{Acknowledgements}
The author is supported  in part by the research program P1-0291 and grants J1-7256 and nd J1-9104
from ARRS, Republic of Slovenia. I wish to thank Antonio Alarc{\'o}n and Rafael Andrist for 
a helpful discussion concerning Corollary \ref{cor:harmonic} and the Schoen-Yau conjecture,
Barbara Drinovec Drnov{\v s}ek for her remarks on the exposition,
Josip Globevnik for the reference to the paper of Bo{\v z}in \cite{Bozin1999IMRN}, 
Frank Kutzschebauch for having proposed to include the material in \S\ref{ss:wild},
and Peter Landweber for his remarks which helped me to improve the language and presentation.


{\bibliographystyle{abbrv} \bibliography{references-Forstneric}}


\newpage
\noindent Franc Forstneri\v c

\noindent Faculty of Mathematics and Physics, University of Ljubljana, Jadranska 19, SI--1000 Ljubljana, Slovenia

\noindent Institute of Mathematics, Physics and Mechanics, Jadranska 19, SI--1000 Ljubljana, Slovenia

\noindent e-mail: {\tt franc.forstneric@fmf.uni-lj.si}

\end{document}